\documentclass[11pt]{amsart}
\usepackage{a4wide}
\usepackage{version}
\usepackage{amsmath}
\usepackage{amsfonts, amssymb, amsthm}
\usepackage{pstricks}
\usepackage{color,mathrsfs,epsfig}
\usepackage{epsfig}
\usepackage{color}

\DeclareMathAlphabet{\mathpzc}{OT1}{pzc}{m}{it}
\newtheorem{thm}{Theorem}

\newtheorem{preremark}[thm]{Remark}

\numberwithin{equation}{section}
\newtheorem{prop}[thm]{Proposition}
\newtheorem{lemma}[thm]{Lemma}

\numberwithin{equation}{section}

\newcommand{\R}{\mathbb R}
\newcommand{\N}{\mathbb N}

\newcommand{\meanbar}[1]{%
\setbox0 = \hbox{$#1 \int$}
\hbox to 0pt{%
\thinspace
\hskip 0.1\wd0
\raise 0.5\ht0
\hbox{%
\lower 0.5\dp0
\hbox{\rule{0.8\wd0}{2\linethickness}}
}%
\hss
}%
}

\makeatletter
\def\@tocline#1#2#3#4#5#6#7{\relax
  \ifnum #1>\c@tocdepth % then omit
  \else
    \par \addpenalty\@secpenalty\addvspace{#2}%
    \begingroup \hyphenpenalty\@M
    \@ifempty{#4}{%
      \@tempdima\csname r@tocindent\number#1\endcsname\relax
    }{%
      \@tempdima#4\relax
    }%
    \parindent\z@ \leftskip#3\relax \advance\leftskip\@tempdima\relax
    \rightskip\@pnumwidth plus4em \parfillskip-\@pnumwidth
    #5\leavevmode\hskip-\@tempdima
      \ifcase #1
       \or\or \hskip 1em \or \hskip 2em \else \hskip 3em \fi%
      #6\nobreak\relax
    \hfill\hbox to\@pnumwidth{\@tocpagenum{#7}}\par% <---- \dotfill -> \hfill
    \nobreak
    \endgroup
  \fi}
\makeatother

\author{Georgiana Chatzigeorgiou}

\title[Regularity for the Fully Nonlinear Parabolic Thin Obstacle Problem]{Regularity for the Fully Nonlinear Parabolic Thin Obstacle Problem}

\begin{document}

\begin{abstract} 
We prove $C^{1,\alpha}$ regularity (in the parabolic sense) for the viscosity solution of a boundary obstacle problem with a fully nonlinear parabolic equation in the interior. Following the method which was first introduced for the harmonic case by L. Caffarelli in 1979, we extend the results of I. Athanasopoulos (1982) who studied the linear parabolic case and the results of E. Milakis and L. Silvestre (2008) who treated the fully nonlinear elliptic case.
\end{abstract}

\let\thefootnote\relax\footnotetext{
{\bf Mathematics Subject Classification (2010).} 35R35; 35R45; 35K55.

{\bf Keywords.} Parabolic thin obstacle problem; Fully non-linear parabolic equations; Free boundary problems; 

\hspace{1,75cm} Regularity of the solution.

{University of Cyprus, Department of Mathematics \& Statistics, P.O. Box 20537, Nicosia, CY- 1678, CYPRUS

\tt chatzigeorgiou.georgiana@ucy.ac.cy}
}

%{\textcolor{red}{Corrections up to here}}

\maketitle

\section{Introduction}

In the present work we intent to study the regularity of the viscosity solution of the following thin obstacle problem in a half-cylinder,
\begin{align} \label{thin_prob}
\begin{cases} 
F(D^2u)-u_t=0, &\ \ \ \ \text{ in } \ \ Q_1^+ \\
u_y \leq 0,&\ \ \ \ \text{ on } \ \ Q_1^* \\
u \geq \varphi,&\ \ \ \ \text{ on } \ \ Q_1^* \\
u_y=0,&\ \ \ \ \text{ on } \ \ Q_1^* \cap \{u>\varphi\} \\
u=u_0,&\ \ \ \ \text{ on } \ \ \partial_p Q_1^+ \setminus Q_1^*
\end{cases}
\end{align}
where, $F$ is a uniformly elliptic operator on $S_n$ with ellipticity constants $\lambda$ and $\Lambda$ and $\varphi: \overline{Q}_1^* \to \R$, $u_0: \partial_p Q_1^+ \setminus Q_1^* \to \R$ are given functions. Function $\varphi$ is the so-called obstacle and $u_0 \geq \varphi$ on $\partial_p Q_1^*$ for compatibility reasons. Our aim is to prove that $u$ is in $H^{1+\alpha}$ up to the flat boundary $Q_1^*$. The main theorem of this paper follows (notations' details can be found in subsection \ref{notations}).

\begin{thm} \label{main_thin}
Let $P_0=(x_0,t_0) \in  Q_{1/2}^*$ be a free boundary point, there exist universal constants $0<\alpha<1, C>0, 0<r<<1$ and an affine function $R_0(X)=A_0+B_0 \cdot (X-(x_0,0))$, where $A_0=u(P_0)$, $B_0=Du(P_0)$ so that
$$|u(X,t)-R_0(X)| \leq C  \left( |X-(x_0,0)|+|t-t_0|^{1/2}\right)^{1+\alpha}, \ \ \text{ for any } \ (X,t) \in Q_r^+(P_0).$$
\end{thm}

The classical obstacle problem as well as the thin obstacle problem are originated in the context of elasticity since model the shape of an elastic membrane which is pushed by an obstacle (which may be very thin) from one side affecting its shape and formation. The same model appears in control theory when trying to evaluate the optimal stopping time for a stochastic process with payoff function. Important cases of obstacle type problems occur when the operators involved are fractional powers of the Laplacian as well as nonlinear operators
since they appear, among others, in the analysis of anomalous diffusion, in quasi-geostrophic flows, in biology modeling flows through semi-permeable membranes for certain osmotic phenomena and when pricing American options regulated by assets evolving in relation to jump processes.

Thin (or boundary) obstacle problem (or Signorini's problem) was extensively studied in the elliptic case. For Laplace equation and more general elliptic PDEs in divergence form the problem can be also understood in the variational form, that is as a problem of minimizing a suitable functional over a suitable convex class of functions which should stay above the obstacle on a part of the boundary (or on a sub-manifold of co-dimension at least $1$) of the domain of definition. The $C^{1,\alpha}$-regularity of the weak solution for the harmonic case was proved first in 1979 by L. Caffarelli in \cite{caf79} who treats also the divergence case for regular enough coefficients. Results for more general divergence-type elliptic operators can be found in \cite{u}. For optimal regularity and regularity of the free boundary in the case of linear elliptic equations we refer to \cite{ac} and \cite{ACS} where the harmonic case is studied and to \cite{nestor}, \cite{GPS16}, \cite{KRS17} for the case of variable coefficients. Similar results exist also for the case of fractional Laplacians. Regularity of the solution for the classical (thick) obstacle problem was studied in \cite{S07}, then via the extension problem introduced in \cite{CS07} the thin obstacle problem was treated in \cite{CSS}. Finally, for fully nonlinear elliptic operators, regularity of the viscosity solution was proved in \cite{ms2} (see also \cite{Fern16}) while for optimal and free boundary regularity the only existing work is \cite{R-OS17}.

The corresponding regularity theory for thin obstacle problems of parabolic type is much less developed. The $C^{1,\alpha}$-regularity of the weak solution was obtained in 1982 by I. Athanasopoulos in \cite{aparab} who studied the case of heat equation and the case of smooth enough linear parabolic equation. The case of more general linear parabolic operators was examined in \cite{U85} and \cite{AU96}. Optimal and free boundary regularity for the caloric case have been obtained very recently in \cite{ACM19} (see also \cite{DGP17}). Finally for the case of parabolic operators of fractional type we refer the reader to \cite{ACM18} and \cite{CF13}.

In this paper our purpose is to combine the techniques of \cite{caf79}, \cite{aparab} and \cite{ms2} adapting them in our fully nonlinear parabolic framework. To achieve this we need up to the boundary H\"older estimates for viscosity solutions of nonlinear parabolic equations with Neumann boundary conditions (as \cite{ms1} is used in \cite{ms2}). This type of estimates have developed recently by the author and E. Milakis in \cite{CM19}.

The paper is organized as follows. In Section 2 we give a list of notations used throughout this text. We discuss also the assumptions we make on the data of our problem and finally we prove a reflection property which is useful in our approach. In Section 3 we examine the semi-concavity properties of our solution. We prove Lipschitz continuity in space variables, a lower bound for $u_t$ and for the second tangential derivatives of $u$ (semi-convexity) and an upper bound for the second normal derivative of $u$ (semi-concavity). All these bounds are universal and hold up to the flat boundary $Q_1^*$. The boundedness of the first and second normal derivatives ensures the existence of $u_{y^+}$ on $Q_1^*$. Our first intention is to prove that $u_{y^+} \leq 0$ on $Q_1^*$ (which apriori holds only in the viscosity sense). To achieve this we use the penalized problem defined and studied in Section 4. Finally in Section 5 we prove the main theorem. To do so we obtain first an estimate in measure (Lemma \ref{measure}) for $u_{y^+}$ on $Q_1^*$ and subsequently we see how such a property can be carried inside $Q_1^+$ (Lemma \ref{Harnack_type}). An iterative application of the above two properties gives the regularity of $u_{y^+}$ on $Q_1^*$ around free boundary points (Lemma \ref{sigmaregularity}) and then our problem can be treated as a non-homogeneous Neumann problem.

\section{Preliminaries}

\subsection{Notations}\label{notations}
We denote $X=(x,y) \in \R^n$, where $x \in \R^{n-1}$ and $y \in \R$ and $P=(X,t)\in \R^{n+1}$, where $X$ are the space variables and $t$ is the time variable. The Euclidean ball in $\R^n$ and the elementary cylinder in $\R^{n+1}$ will be denoted by
$$B_r(X_0) := \{ X \in \R^n : |X-X_0| < r\}, \ \ Q_r(X_0,t_0) := B_r (X_0) \times (t_0 - r^2, t_0]$$
respectively. We define the following half and thin-balls in $\R^n$, for $r>0, x_0 \in \R^{n-1}$
$$B_r^+(x_0):= B_r(x_0,0) \cap \{y >0 \},  \ \ B_r^*(x_0):= B_r(x_0,0) \cap \{y =0 \}$$
and the following half and thin-cylinders in $\R^{n+1}$, for $r>0, x_0 \in \R^{n-1}_+, t_0 \in \R$
$$Q_r^+(x_0,t_0):= Q_r(x_0,0,t_0) \cap \{y >0 \},  \ \ Q_r^*(x_0,t_0):= Q_r(x_0,0,t_0) \cap \{y =0 \}.$$
Note that, $\Omega^\circ, \overline{\Omega}, \partial \Omega$ will be the interior, the closure and the boundary of the domain $\Omega \subset \R^{n+1}$, respectively, in the sense of the Euclidean topology of $\R^{n+1}$. We define also the parabolic interior to be,
$$int_p(\Omega ) := \{ (X,t) \in \R ^{n+1} : \  \text{ there exists } \  r>0 \  \text{ so  that } \ Q_r^\circ(X,t) \subset \Omega \}$$
and the parabolic boundary, $\partial _p (\Omega ) := \overline{\Omega } \setminus int_p(\Omega ).$ Let us also define the parabolic distance for $P_1=(X,t),\ P_2=(Y,s) \in \R^{n+1} $, $p(P_1,P_2) := \max \{|X-Y|,|t-s|^{1/2}\}$. Note that in this case $Q_r(P_0)$ will be the set $\{ P \in \R^{n+1} : p(P,P_0)<r, t<t_0 \}$.

Next we define the corresponding parabolic H\"older spaces. For a function $f$ defined in a domain $\Omega \subset \R^{n+1}$ we set,

$$[f]_{\alpha; \Omega} := \sup _ {P_1, P_2 \in \overline{\Omega}, P_1 \neq P_2} \frac{|f(P_1) - f(P_2)|}{p(P_1,P_2)^{\alpha}}, \ \ \langle f \rangle _{\alpha+1; \Omega} := \sup _ {(X,t_1), (X,t_2) \in \overline{\Omega} \atop t_1 \neq t_2} \frac{|f(X,t_1) - f(X,t_2)|}{|t_1-t_2|^{\frac{\alpha+1}{2}}}.$$

Then we say that,
\begin{enumerate}
\item[$\bullet$] $f \in H^\alpha (\overline{\Omega})$ if
$||f||_{H^{\alpha} (\overline{\Omega}) } := \sup _{\overline{\Omega}} |f| + [f]_{\alpha; \Omega}  < + \infty .$
\item[$\bullet$] $f \in H^{\alpha+1} (\overline{\Omega})$ if 
$$||f||_{H^{\alpha+1} (\overline{\Omega}) } := \sup _{\overline{\Omega}} |f| +  \sum_{i=1}^n  \sup _{\overline{\Omega}} |D_if| + \sum_{i=1}^n [D_if]_{\alpha; \Omega} +\langle f \rangle _{\alpha+1; \Omega}  < + \infty.$$
\item[$\bullet$] $f \in H^{\alpha+2} (\overline{\Omega})$ if
\begin{align*}
||f||_{H^{\alpha+2} (\overline{\Omega}) } :=& \sup _{\overline{\Omega}} |f| +  \sum_{i=1}^n  \sup _{\overline{\Omega}} |D_if| +\sup _{\overline{\Omega}} |f_t|+ \sum_{i,j=1}^n  \sup _{\overline{\Omega}} |D_{ij}^2f|\\
&+[f_t]_{\alpha; \Omega}
+ \sum_{i,j=1}^n [D_{ij}^2f]_{\alpha; \Omega} +\sum_{i=1}^n \langle D_i f \rangle _{\alpha+1; \Omega}  < + \infty .
\end{align*}
\end{enumerate}
Due to the nonlinear character of our problem, we will mainly prove $H^{\alpha+1}$-regularity results in the punctual sense at a point. We say that $u$ is punctually $H^{\alpha+1}$ at a point $P_1 \in \overline{\Omega}$ if there exists $R_{1;P_1}(X)= A_{P_1}+ B_{P_1} \cdot (X-X_1)$, where $A_{P_1} \in \R$ and $B_{P_1} \in \R^{n}$ and some cylinder $\overline{Q}_{r_1}(P_1) \subset \Omega$, so that for any $0<r<r_1$,
$$|u(X,t) - R_{1;P_1}(X)| \leq K \ r^{1+\alpha}, \ \ \text{ for every } \ (X,t) \in \overline{Q}_{r}(P_1)$$
for some constant $K>0$. 

Finally, $S_n$ denotes the class of symmetric $n \times n$ real matrices.

\subsection{Problem Set-up}
We consider that the solution $u$ of (\ref{thin_prob}) can be recovered as the minimum viscosity supersolution of
\begin{align} \label{perron_thin_prob}
\begin{cases} 
F(D^2v)-v_t\leq 0, &\ \ \ \ \text{ in } \ \ Q_1^+ \\
v_y \leq 0,&\ \ \ \ \text{ on } \ \ Q_1^* \\
v \geq \varphi,&\ \ \ \ \text{ on } \ \ Q_1^* \\
v\geq u_0,&\ \ \ \ \text{ on } \ \ \partial_p Q_1^+ \setminus Q_1^*
\end{cases}
\end{align}
with $u_t$ locally bounded by above in $Q_1^+$ (note that under suitable assumptions on $F$ we have that $u_t$ does exist in $Q_1^+$ once $F(D^2u)-u_t=0$ in $Q_1^+$ in the viscosity sense).

To get the desired regularity we make the following assumptions on $F$ and $u_0$.
\\ $\bullet$ \textbf{Assumptions on} $F$. First we assume that $F$ is convex on $S_n$ so we have interior $H^{2+\alpha}$-estimates for the viscosity solutions (see \cite{Wang2}). Moreover considering the following extension of $F$ in $\R^{n \times n}$
$$F(M) = F \left( \frac{M+M^\tau}{2}\right), \ \ \text{ for } \ \ M \in \R^{n \times n}$$
we assume that $F$ is continuously differentiable in $\R^{n^2}$ and we denote by $F_{ij}:=\frac{\partial F}{\partial m_{ij}}$. We can easily see that $F_{ij}(M)= F_{ji}(M)$ for any $M$. Indeed, let $H^{ij}$ denote the matrix with elements
\begin{align*}
\left(H^{ij}_h \right)_{kl}= \begin{cases}  
0, \ \  \text{ if } \ \ k \neq i \ \text{ or } \ j \neq l, \\
h, \ \  \text{ if } \ \ k = i \ \text{ and } \ j = l
\end{cases} 
\end{align*} 
where $h \in \R$ and observe that $\left(H^{ij}_h\right)^\tau=H^{ji}_h$ then
\begin{align*}
F_{ij}(M)&=  \lim_{h \to 0} \frac{F\left(\frac{M+H^{ij}_h+\left( M+H^{ij}_h\right)^\tau}{2}\right)-F(M)}{h} 
= \lim_{h \to 0}  \frac{F\left(\frac{M+H^{ji}_h+\left( M+H^{ji}_h\right)^\tau}{2}\right)-F(M)}{h}=F_{ji}(M).
\end{align*}
We suppose also that $F_{in}=0$, for any $i=1,\dots , n-1$ (then $F_{ni}=0$ as well). Finally, we assume for convenience that $F(O)=0$ which can be easily removed (subtracting a suitable paraboloid).
\\ $\bullet$ \textbf{Assumptions on} $u_0$. Note that we intend to examine the regularity up to flat boundary $Q_1^*$ (and not up to $\partial_pQ_1^+ \setminus Q_1^*$) thus we may assume that $u_0>\varphi$ on $\partial_pQ_1^*$. Therefore if $v \in C\left( \overline{Q}_1^+\right)$ is the viscosity solution of
\begin{align} \label{rho_def}
\begin{cases} 
F(D^2v)-v_t=0, &\ \ \ \ \text{ in } \ \ Q_1^+ \\
v_y = 0,&\ \ \ \ \text{ on } \ \ Q_1^* \\
v=u_0,&\ \ \ \ \text{ on } \ \ \partial_p Q_1^+ \setminus Q_1^*
\end{cases}
\end{align}
then due to the continuity of $v$ and $\varphi$ and the compactness of $\partial_pQ_1^*$ we see that there exists some $0<\rho<1$ so that $v>\varphi$ on $Q_1^* \setminus Q_{1-\rho}^*$. Then using an ABP-type estimate (see Theorem 5 in \cite{CM19}) we get that $u>\varphi$ on $Q_1^* \setminus Q_{1-\rho}^*$ thus $u_y=0$ on $Q_1^* \setminus Q_{1-\rho}^*$, in the viscosity sense.
\\ $\bullet$ \textbf{Assumptions on} $\varphi$. We assume that $\varphi \in H^{2+\alpha_0} \left( Q_1^* \right)$. 

We denote by $\Delta^*:= \{(x,t) \in Q_1^* : u(x,0,t)=\varphi(x,t)\}$ the \textit{contact set}, by $\Omega^*:= \{(x,t) \in Q_1^* : u(x,0,t)>\varphi(x,t)\}$ the \textit{non-contact set} and by $\Gamma=\partial \Delta^* \cap Q_1^*$ the \textit{free boundary}. We assume that $\Delta^* \neq \emptyset$ since otherwise we would have a Neumann boundary value problem for which the regularity is known (see \cite{CM19}). Note that around the points of int$\left(\Delta^*\right)$ and around the points of $\Omega^*$ we can treat our problem as Dirichlet or Neumann problem respectively. Finally, we denote by $K:=||u||_{L^\infty\left(Q_1^+\right)}+ ||\varphi||_{H^{2+\alpha_0}\left(Q_1^*\right)}$ and in the following a constant $C>0$ that depends only on $K, n, \lambda , \Lambda$ and $\rho$ will be called \textit{universal}. 

\subsection{Reflection Principle}

Here we show a reflection property which will be useful in several times in our approach. We remark that since $F_{in}=F_{ni}=0$ for any $i=1, \dots , n-1$ then for $M=(m_{ij}) \in \R^{n \times n}$ if we denote by $\bar{M}$ the matrix with elements
\begin{align*}
\bar{m}_{ij}:= \begin{cases} m_{ij}, &\text{ if } \ i,j<n \ \ \text{ or } \ \ i=j=n \\
-m_{ij}, &\text{ if } \ i<n \text{ and } j=n \ \ \text{ or } \ \ i=n \text{  and } j<n \end{cases}
\end{align*}
we have that $F(M)=F(\bar{M})$. Observe that Pucci's extremal operators have this property as well. Indeed, $M$ and $\bar{M}$ have the same eigenvalues since,
\begin{align*} 
 \det(\bar{M}-lI_n) = \det \left(  \begin{matrix}
 M_{n-1} -lI_{n-1} &  d^\tau \\
d & m_{nn}-l
\end{matrix} \right) = \det \left(  \begin{matrix}
 M_{n-1} -lI_{n-1}  &  -d^\tau \\
-d & m_{nn}-l
\end{matrix} \right) = \det(M-lI_n)
\end{align*}
where $d:=(-m_{n1},\dots,-m_{nn-1})$, hence $\mathcal{M}^\pm (\bar{M},\lambda, \Lambda )=\mathcal{M}^\pm (M,\lambda, \Lambda )$.

\begin{prop}(Reflection Principle). \label{reflection}
Let $u \in C(Q_1^+ \cup Q^*_1)$ and satisfies $F(D^2u)-u_t\leq 0$ in $Q_1^+$ and $u_y \leq 0$ on $Q^*_1$ in the viscosity sense. Consider the reflected function, 
\begin{align*}
u^*(x,y,t) = \begin{cases} u(x,y,t), &\ \ \ \text{ if } \ \ \ y \geq 0 \\ u(x,-y,t), &\ \ \ \text{ if } \ \ \ y <0 \end{cases}, \ \ \ \text{ for } \ \ (X,t) \in Q_1.
\end{align*}
Then $F(D^2u^*)-u^*_t\leq 0$ in $Q_1$ in the viscosity sense.
\end{prop}

\begin{proof}
We observe that $u^* \in C(Q_1)$  and that $F(D^2u^*)-u^*_t\leq 0$ in $Q_1^+$. Also it can be easily verified $F(D^2u^*)-u^*_t\leq 0$ in $Q_1^-$ (regarding the observation we made above). To get that this is true in $Q_1$ as well it remains to study what happens across $Q^*_1$. To do so we approximate by suitable supersolutions, by considering
$$v_\gamma (X,t) := u^*(X,t) - \gamma |y|$$
for $\gamma >0$. Then we have that $F(D^2v_\gamma )-\left(v_\gamma\right)_t\leq 0$ in $Q_1^+ \cup Q_1^-$ and we will show that $v_\gamma$ cannot be touched by below by any test function at any point of $Q^*_1$. Indeed, let $\phi $ be a test function in $Q_1$ that touches $v_\gamma$ by below at some point $P_0=(x_0,0,t_0) \in Q^*_1$. Our purpose is to use the viscosity Neumann condition to get a contradiction. We have that $\phi (X,t) + \gamma y$ touches $u$ by below at $P_0$ in some $Q_\rho^+(P_0) \subset Q_1^+$. Then $\phi_y(P_0) +\gamma \leq 0$, that is $\phi_y(P_0) \leq -\gamma < 0.$ But on the other hand, $\phi (X,t) - \gamma y$ touches $u^*$ by below at $P_0$ in some $Q_\rho^-(P_0) \subset Q_1^-$. A change of variables implies that $u(X,t) \geq \phi (x',-y,t) + \gamma y$, for $(X,t) \in Q_\rho^+(P_0)$. Then $-\phi_y(P_0) + \gamma \leq 0$, that is $\phi_y(P_0) \geq \gamma > 0$, a contradiction. Therefore such a test function cannot exist. 
\par Consequently, since  $F(D^2v_\gamma )-\left(v_\gamma\right)_t\leq 0$ in $Q_1^+ \cup Q_1^-$ and no test function can touch $v_\gamma $ by below at any point of $Q^*_1$ in a neighborhood in $Q_1$, that is  $F(D^2v_\gamma )-\left(v_\gamma\right)_t\leq 0$  in $Q_1$ in the viscosity sense. Finally, we observe that, $|v_\gamma - u^*| =|\gamma||y| \leq |\gamma| \to 0$ as $\gamma \to 0$, which means that $v_\gamma \to u^*$, as $\gamma \to 0$ uniformly in $Q_1$. So, we can consider for $k \in \mathbb{N}$ the sequence $\{ v_{\frac{1}{k}} \}$ and use the closedness of viscosity supersolutions to complete the proof. 
\end{proof}

Note that an analogous result holds for subsolutions. That is, if $v \in C(Q_1^+ \cup Q^*_1)$ which satisfies $F(D^2v)-v_t\geq 0$ in $Q_1^+$ and $v_y \geq 0$ on $Q^*_1$ in the viscosity sense, then $F(D^2v^*)-v^*_t\geq 0$ in $Q_1$.

\section{Semi-concavity properties}
In this section we obtain bounds for the first and second derivatives of the solution. A first application of these bounds will ensure the existence of $u_{y^+}$ on $Q_1^*$.

\begin{prop} \label{semiconcavity}
For any $0<\delta<1$,
\begin{enumerate}
\item[(A)] $|u_{x_i}|, |u_y| \leq C$, in $Q_{1-\delta}^+$, for any $i=1,\dots,n-1$
\item[(B)] $u_{x_ix_i}, u_t\geq -C$, in $Q_{1-\delta}^+$, for any $i=1,\dots,n-1$
\item[(C)] $u_{yy} \leq C$, in $Q_{1-\delta}^+$
\end{enumerate}
where the constant $C>0$ depends only on $K, n, \lambda , \Lambda, \rho$ and $\delta$.
\end{prop}

Note that since $F$ is convex, we have that $u_{x_ix_j}$ and $u_t$ exist in $Q_1^+$ in the classical sense by interior estimates (see \cite{Wang2}) .

\begin{proof}
$\ $

For \textit{(A)}, we thicken the obstacle $\varphi$. First, we extend $\varphi$ as a solution inside $Q_1^+$ and $Q_1^-$ (following the idea of Theorem 1(a) in \cite{ac}, see also Proposition 2.1 in \cite{Fern16}), that is we consider the viscosity solutions of the Dirichlet problems
\begin{align*} 
\begin{cases}
F(D^2\tilde{\varphi}) -\tilde{\varphi}_t =0, &\ \text{ in } \ Q_1^+ \\
\tilde{\varphi}=\varphi , &\ \text{ on } \ Q_1^* \\
\tilde{\varphi}=-||u||_{L^\infty\left(Q_1^+\right)}, &\ \text{ on } \ \partial_pQ_1^+ \setminus Q_1^*
\end{cases} \ \ \text{ and } \ \
\begin{cases} 
F(D^2\tilde{\varphi}) -\tilde{\varphi}_t =0, &\ \text{ in } \ Q_1^- \\
\tilde{\varphi}=\varphi , &\ \text{ on } \ Q_1^* \\
\tilde{\varphi}=-||u||_{L^\infty\left(Q_1^+\right)}, &\ \text{ on } \ \partial_pQ_1^- \setminus Q_1^*.
\end{cases}
\end{align*}
For any $0<\delta<1$ and since $\varphi$ is smooth enough we obtain, using Theorem 12 of \cite{CM19}, that $\tilde{\varphi}$ is Lipschitz in $Q_{1-\frac{\delta}{2}}$ with a constant that depends only on $K, n, \lambda , \Lambda$ and $\delta$. Moreover using maximum principle we can obtain that $u^* \geq \tilde{\varphi}$ in $Q_1$, where $u^*$ denotes the even reflection of $u$ in $y$ inside $Q_1$. Finally, Proposition \ref{reflection} ensures that $F(D^2u^*) -u^*_t \leq 0$ in  $Q_1$ and that $F(D^2u^*) -u^*_t =0$ in  $Q_1 \cap \{u^*>\tilde{\varphi}\}$ in the viscosity sense. Therefore $u^*$ satisfies a thick obstacle problem in $Q_1$ with obstacle $\tilde{\varphi}$ which is Lipschitz in $Q_{1-\frac{\delta}{2}}$. In particular, we get that $u^* \in H^1 \left(Q_{1-\delta}\right)$ with a constant that depends only on $K, n, \lambda , \Lambda$ and $\delta$ (see \cite{Sh}, \cite{PS07}) which gives \textit{(A)}.

For \textit{(B)}, we denote by $d:=\min \{\rho,\delta\}$ and we consider the set $\tilde{Q}^+:= Q_{1-\frac{d}{3}}^+ \setminus Q_{1-\frac{2d}{3}}^+$. We observe that $u_y=0$ on $\tilde{Q}^*$ in the viscosity sense, since $\tilde{Q}^* \subset Q_1^* \setminus Q_{1-\rho}^*$. Thus up to the boundary $H^{2+\alpha}$-estimates (see Theorem 23 in \cite{CM19}) can be applied in $\tilde{Q}^+$ and we get $H^{\alpha}$-estimates for $u_{x_ix_i}$ and $u_t$ on $\partial_pQ^+_{1-\frac{d}{2}} \setminus Q^*_{1-\frac{d}{2}}$. In particular we have uniform bounds for the corresponding difference quotients, that is,
\begin{equation} \label{u_xx_bound}
\frac{u(x+he_i,y,t)+u(x-he_i,y,t)-2u(x,y,t)}{h^2}\geq -C
\end{equation}
where $\{e_i\}_{1\leq i \leq n}$ is the normal basis of $\R^n$ and
\begin{equation} \label{u_t_bound}
\frac{u(x,y,t-h)-u(x,y,t)}{h}\geq -C
\end{equation}
for $(X,t) \in \partial_pQ^+_{1-\frac{d}{2}}\setminus Q^*_{1-\frac{d}{2}}$, $h>0$ small enough (depending only on $d$) and $C>0$ depends only on $K, n, \lambda , \Lambda, \rho$ and $\delta$.

We study (\ref{u_xx_bound}) first in order to bound $u_{x_ix_i}$, for $i=1,\dots,n-1$. We observe that 
$$v(x,y,t):=\frac{u(x+he_i,y,t)+u(x-he_i,y,t)}{2}+Ch^2\geq u(x,y,t), \ \ \text{ on } \ \partial_pQ^+_{1-\frac{d}{2}}\setminus Q^*_{1-\frac{d}{2}}.$$
Moreover, for $(x,t) \in Q^*_{1-\frac{d}{2}}$,
\begin{align*}
v(x,0,t)&=\frac{u(x+he_i,0,t)+u(x-he_i,0,t)}{2}+Ch^2 \\
&\geq \frac{\varphi(x+he_i,t)+\varphi(x-he_i,t)}{2}+Ch^2 \geq \varphi (x,t)
\end{align*}
changing $C$ if necessary depending on $K$. We observe also that the convexity of $F$ ensures that $F(D^2v)-v_t \leq 0$ in $Q^+_{1-\frac{d}{2}}$ in the viscosity sense. Finally note that $v_y \leq 0$ on $Q^*_{1-\frac{d}{2}}$ in the viscosity sense (which can be obtained as Proposition 11 in \cite{CM19}). That is $v$ is a viscosity supersolution of (\ref{perron_thin_prob}) in $Q^+_{1-\frac{d}{2}}$, thus $v \geq u$ in $Q^+_{1-\frac{d}{2}}$. Therefore
$$\frac{u(x+he_i,y,t)+u(x-he_i,y,t)-2u(x,y,t)}{h^2}\geq -C$$
in $Q^+_{1-\frac{d}{2}}$ and $C>0$ depends only on $K, n, \lambda , \Lambda, \rho$ and $\delta$. Next we study (\ref{u_t_bound}) in a similar way in order to bound $u_{t}$. Observe that 
$$w(x,y,t):=u(x,y,t-h)+Ch\geq u(x,y,t), \ \ \text{ on } \ \partial_pQ^+_{1-\frac{d}{2}} \setminus Q^*_{1-\frac{d}{2}}.$$
Moreover, for $(x,t) \in Q^*_{1-\frac{d}{2}}$,
\begin{align*}
w(x,0,t)&=u(x,0,t-h)+Ch \geq \varphi(x,t-h)+Ch \geq \varphi (x,t)
\end{align*}
changing $C$ if necessary depending on $K$. Finally, note that $F(D^2w)-w_t =0$ in $Q^+_{1-\frac{d}{2}}$ and $w_y \leq 0$ on $Q^*_{1-\frac{d}{2}}$. That is $w$ is a viscosity supersolution of (\ref{perron_thin_prob}) in $Q^+_{1-\frac{d}{2}}$, thus $w \geq u$ in $Q^+_{1-\frac{d}{2}}$. Therefore
$$\frac{u(x,y,t-h)-u(x,y,t)}{h}\geq -C$$
in $Q^+_{1-\frac{d}{2}}$ and $C>0$ depends only on $K, n, \lambda , \Lambda, \rho$ and $\delta$.

For \textit{(C)}, we will use \textit{(B)} and the equation. Define
$$a_{ij}(X,t):=\int_0^1 F_{ij}\left( hD^2u(X,t)\right) \ dh$$
and we observe that $\frac{d}{dh} \left[ F \left( hD^2u(X,t)\right) \right]=\sum_{i,j=1}^n F_{ij}\left( hD^2u(X,t)\right) \ u_{x_ix_j}(X,t)$. That is,
\begin{align*}
\sum_{i,j=1}^n a_{ij}(X,t) \ u_{x_ix_j}(X,t)&=\int_0^1 \sum_{i,j=1}^n F_{ij}\left( hD^2u(X,t)\right) \ u_{x_ix_j}(X,t) \ dh =F \left( D^2u(X,t)\right).
\end{align*}
Thus, $\sum_{i,j=1}^n a_{ij}(X,t) \ u_{x_ix_j}(X,t)-u_t(X,t)=0$ in $Q_1^+$. Also, we have that $a_{ij}=a_{ji}$ and that $a_{in}=a_{ni}=0$, for any $1\leq i \leq n-1$, from our assumptions on $F$. Additionally we may observe that using the ellipticity of $F$ we have that for any $M \in S_n$ and $h>0$
$$\lambda h \leq F \left(M+H^{ii}_h\right) -F(M) \leq \Lambda h$$
so taking $h \to 0^+$ we have that $\lambda  \leq F_{ii}(M) \leq \Lambda$. In particular, $\lambda  \leq a_{ii}(X,t) \leq \Lambda$, for any $(X,t) \in Q_1^+$, $i=1,\dots ,n$. So if $A_{n-1}(X,t) := \left( a_{ij}(X,t) \right)_{i,j=1,\dots,n-1} \in S_{n-1}$ we have
\begin{align*}
a_{nn}(X,t) u_{yy}(X,t) &= -\sum_{i,j=1}^{n-1} a_{ij}(X,t) \ u_{x_ix_j}(X,t)+u_t(X,t)\\
&=- \text{tr}  \left( A_{n-1}(X,t) \ D^2_{n-1}u(X,t) \right) +u_t(X,t) \\
&=-\text{tr}  \left[ A_{n-1}(X,t)\left( \ D^2_{n-1}u(X,t)+C\text{I}_{n-1} \right) \right] + \text{tr}  \left(C A_{n-1}(X,t)\right) +u_t(X,t)
\end{align*}
where $C>0$ is the constant in \textit{(B)}, thus 
$$\text{tr}  \left[ A_{n-1}(X,t)\left( \ D^2_{n-1}u(X,t)+C\text{I}_{n-1} \right) \right]=\text{tr}  \left[ A_{n-1}(X,t)\right] \text{tr}\left[ \ D^2_{n-1}u(X,t)+C\text{I}_{n-1} \right]\geq 0$$
and $\text{tr}  \left( C A_{n-1}(X,t)\right) = C \sum_{i=1}^{n-1}a_{ii}(X,t)\leq C \Lambda (n-1)$, $a_{nn}(X,t)\geq \lambda$. Hence we have that
$$u_{yy}\leq C \ \frac{\Lambda (n-1)+1}{\lambda}, \ \text{ in } \ Q_{1-\delta}^+.$$
\end{proof}

For any $(x,t) \in Q_1^*$ we define
$$\sigma(x,t):=\lim_{y \to 0^+} u_y(x,y,t).$$
Note that Proposition \ref{semiconcavity} ensures the existence of the above limit for any $(x,t) \in Q_1^*$. Indeed, we consider the function $v(X,t)=u_y(X,t)-Cy$, for $(X,t) \in Q_1^+$. Then using \textit{(A)} and \textit{(C)} of Proposition \ref{semiconcavity} we obtain that $v>-
2C$ and $v_y=u_{yy}-C\leq 0$ in $Q_{1-\delta}^+$, that is, $v$ is monotone decreasing in $y$ and bounded by below, thus $\lim_{y \to 0^+} v(x,y,t)$ exists for $(x,t) \in Q_{1-\delta}^*$, for any $0<\delta<1$. 

Furthermore we remark that the existence of the above limit ensures the existence of 
\\$\lim_{y \to 0^+} \frac{u(x,y,t)-u(x,0,t)}{y}$, that is $u_{y^+}$ exists on $Q_1^*$ and equals to $\sigma$ (note also that $u_{y}$ is continuous in $y$ up to $Q_1^*$). Thereafter the viscosity condition $u_y \leq 0$ on $Q_1^*$ suggests that one should have 
\begin{equation} \label{sigmasign1}
\sigma \leq 0, \ \text{ on } \ Q_1^*.
\end{equation}
Although we know that $u_{y^+}=\sigma$ on $Q_1^*$ in the classical sense, we cannot use the viscosity condition to get (\ref{sigmasign1}) since we do not know if $u_{y^+}$ is continuous in $(x,0,t)$. To obtain (\ref{sigmasign1}) we use a penalization technique introduced in the next section.

\section{A penalized problem}

We focus now on showing (\ref{sigmasign1}) by approximating $u$ by suitable classical solutions. So for any $k \in \N$ we consider the penalized problem
\begin{align} \label{pen_prob}
\begin{cases} 
F\left( D^2u^{(k)}\right)-\left(u^{(k)}\right)_t=0, &\ \ \ \ \text{ in } \ \ Q_1^+ \\
\left(u^{(k)}\right)_y =-k \left( \varphi - u^{(k)}\right)^+:=g^{(k)}, &\ \ \ \ \text{ on } \ \ Q_1^* \\
u^{(k)}=u_0,&\ \ \ \ \text{ on } \ \ \partial_p Q_1^+ \setminus Q_1^*.
\end{cases}
\end{align}

Note that (\ref{pen_prob}) is not a free boundary problem. Using ABP-estimate and a barrier argument we obtain estimates for $u^{(k)}$ and $g^{(k)}$ (Lemmata \ref{u^k_est} and \ref{g^k_est}) which are independent of $k$. Then we will be able to treat (\ref{pen_prob}) as a non-homogeneous Neumann problem and, using suitable H\"older estimates, we obtain the uniform convergence of $u^{(k)}$ to $u$ (Proposition \ref{u^k_convergence}) and the existence of $\left(u^{(k)}\right)_y$ in the classical sense (Lemma \ref{u^k_regularity}). This last property means that the viscosity condition for $u^{(k)}$ holds in the classical sense. This makes the penalized problem very useful in proving (\ref{sigmasign1}) (see Lemma \ref{sigmasign} below). Note also that for any $k \in \N$, we have that $u^{(k)}>\varphi$ on $Q_1^* \setminus Q_{1-\rho}^*\ $ by  comparing $u^{(k)}$ with the solution $v$ of (\ref{rho_def}) (see Theorem 5 in \cite{CM19}).

\begin{lemma}[Independent of $k$ estimate for $u^{(k)}$]  \label{u^k_est} 
For any $k \in \mathbb{N}$,
\begin{equation}\label{u^k_est1}
||u^{(k)}||_{L^\infty\left(Q_1^+\right)} \leq \max \{ ||u||_{L^\infty\left(\partial_p Q_1^+ \setminus Q_1^*\right)}, \ ||\varphi||_{L^\infty\left(Q_1^*\right)} \}.
\end{equation}
\end{lemma}

\begin{proof}
First, by Theorem 5 of \cite{CM19} we have
$$\inf_{Q_1^+}u^{(k)}\geq \inf_{\partial_p Q_1^+ \setminus Q_1^*} u^{(k)}=\inf_{\partial_p Q_1^+ \setminus Q_1^*} u$$
since $\left(u^{(k)}\right)_y \leq 0$ in the viscosity sense on $Q_1^*$. Hence it remains to bound $\sup_{Q_1^+} u^{(k)}$.

Assume that 
$$\sup_{Q_1^+}u^{(k)}>\sup_{\partial_p Q_1^+ \setminus Q_1^*} u$$
and let $(X_0,t_0) \in \overline{Q}_1^+$ be such that $u^{(k)}(X_0,t_0)=\sup_{Q_1^+}u^{(k)}=:M$. From maximum principle (see \cite{Wang1}, Corollary 3.20) we know that $||u^{(k)}||_{L^\infty\left(\overline{Q}_1^+\right)} \leq||u^{(k)}||_{L^\infty\left(\partial_p Q_1^+\right)}$, thus we can choose $(X_0,t_0)=(x_0,0,t_0) \in Q_1^*$. Then by Hopf's lemma we obtain that $u^{(k)}_y(x_0,t_0)<0$ in the viscosity sense. Therefore $-k\left(\varphi(x_0,t_0)-u^{(k)}(x_0,0,t_0) \right)<0$, that is $M=u^{(k)}(x_0,0,t_0)<\varphi(x_0,t_0) \leq ||\varphi||_{L^\infty\left(Q_1^*\right)}$.
\end{proof} 

\begin{lemma}[Independent of $k$ estimate for $g^{(k)}$]  \label{g^k_est} 
For any $k \in \mathbb{N}$,
\begin{equation}\label{g^k_est1}
||g^{(k)}||_{L^\infty\left(Q_1^*\right)} \leq C \left( K, n, \lambda , \Lambda,\rho \right).
\end{equation}
\end{lemma}

\begin{proof}
Note that $g^{(k)} \leq 0$ on $Q_1^*$, so we need to obtain only a lower bound. Let $(x_0,t_0) \in \overline{Q}_1^*$ be such that $g^{(k)}(x_0,t_0)= \min_{\overline{Q}_1^*} g^{(k)}$ and we may assume that $g^{(k)}(x_0,t_0)<0$. Recall also that $u^{(k)}>\varphi$ on $Q_1^* \setminus Q_{1-\rho}^*$ which implies that $g^{(k)}=0$ on $Q_1^* \setminus Q_{1-\rho}^*$, that is, $(x_0,t_0) \in Q_{1-\rho}^*$.

We intend to turn the obstacle $\varphi$ into a suitable test function that touches $u^{(k)}$ by below at $(x_0,t_0)$ and then to use the viscosity condition $\left(u^{(k)}\right)_y=g^{(k)}$ to bound $g^{(k)}(x_0,t_0)$. We denote by
$M:=\inf_{Q_1^+}u-\sup_{Q_1^*}\varphi$ and observe that $M\leq 0$, indeed 
$$\inf_{Q_1^+}u \leq \inf_{Q_1^*}u \leq u(x^*,0,t^*)=\varphi(x^*,t^*) \leq \sup_{Q_1^*}\varphi$$
where $(x^*,t^*)$ is any point of $\Delta^*$. Keep also in mind that by Lemma \ref{u^k_est}, $M \leq \inf_{Q_1^+}u^{(k)}-\sup_{Q_1^*}\varphi$. We consider $b$ to be the solution of the following Dirichlet boundary value problem
\begin{align*}
\begin{cases}
\mathcal{M}^-\left(D^2b, \frac{\lambda}{n}, \Lambda \right) -b_t= (\Lambda n+1)\ ||\varphi||_{H^{2+\alpha}\left(Q_1^*\right)}, &\ \text{ in } \ \ Q^+_\rho \\
b=M, &\ \text{ on } \ \ \partial_pQ^+_\rho \setminus Q_\rho^* \\
b=0, &\ \text{ on } \ \ Q^*_{\rho/2} \\
b(x,0,t) = \frac{2M}{\rho} \left( \max\left\lbrace |x|,|t|^{\frac{1}{2}}\right\rbrace-\frac{\rho}{2}\right), &\ \text{ on } \ \ Q_\rho^* \setminus Q^*_{\rho/2}.
\end{cases}
\end{align*}
Note that $\frac{2M}{\rho} \left( \max\left\lbrace |x|,|t|^{\frac{1}{2}}\right\rbrace-\frac{\rho}{2}\right) =0$ on $\partial_pQ_{\rho/2}^*\ $ and $\ \frac{2M}{\rho} \left( \max\left\lbrace |x|,|t|^{\frac{1}{2}}\right\rbrace-\frac{\rho}{2}\right)=M$ on $\partial_pQ_{\rho}^*$. Hence the Dirichlet data on $\partial_pQ_\rho^+$ is a continuous function. Moreover applying regularity results for Dirichlet problems in $Q_{\rho/2}^+$, we obtain that $b \in H^{1+\alpha}\left(\overline{Q}^+_{\rho/4}\right)$ with the corresponding estimate depending only on $\rho, n, \lambda , \Lambda , K$, in particular, $|Db(0,0)| \leq C \left( K, n, \lambda , \Lambda,\rho \right)$.

Next, we consider the function
$$\Phi(X,t)=u^{(k)}(x_0,0,t_0)-\varphi (x_0,t_0)+\varphi(x,t)+b \left( (X,t)-(x_0,0,t_0)\right)$$ 
for $(X,t) \in Q_\rho^+(x_0,t_0) \subset Q_1^+$. We have that $\Phi(x_0,0,t_0)=u^{(k)}(x_0,0,t_0)$. On $\partial_p Q^+_\rho(x_0,t_0) \setminus Q^*_\rho(x_0,t_0)$, $\Phi(X,t)\leq \inf_{Q_1^+}u^{(k)} \leq u^{(k)}(X,t)$, since $g^{(k)}(x_0,t_0)<0$ and $b=M$. Also on $Q^*_\rho(x_0,t_0)$, $\Phi(x,0,t) \leq u^{(k)}(x,0,t)-\varphi (x,t)+\varphi (x,t)=u^{(k)}(x,0,t)$, using that $b \leq 0$ on $\partial_pQ_\rho^+$, $g^{(k)}(x_0,t_0) \leq g^{(k)}(x,t)$ for any $(x,t) \in \overline{Q}^*_1$ and $g^{(k)}(x_0,t_0) <0$. That is we have that $\Phi \leq u^{(k)}$ on $\partial_p Q^+_\rho(x_0,t_0)$. Note also that if we extend $\varphi$ in $Q_1^+$ by $\varphi(X,t)=\varphi(x,t)$ and $l_i$, $i=1,\dots,n$ denote the eigenvalues of $D^2 \varphi \in S_n$ then
\begin{align*}
\mathcal{M}^-\left(D^2\varphi, \frac{\lambda}{n}, \Lambda \right) -\varphi_t \geq -\Lambda n ||D^2 \varphi||_\infty -  |\varphi_t|  \geq -(\Lambda n +1) ||\varphi||_{H^{2+\alpha}\left(Q_1^*\right)}.
\end{align*}
That is, $\mathcal{M}^-\left(D^2b+D^2\varphi, \frac{\lambda}{n}, \Lambda \right)-b_t -\varphi_t \geq 0.$ Thus, $u^{(k)}-\Phi \in \overline{S}_p \left( \frac{\lambda}{n}, \Lambda \right)$ in $Q_\rho^+(x_0,t_0)$. Applying maximum principle we have that $\Phi \leq u^{(k)}$ in $Q_\rho^+(x_0,t_0)$. In other words, $\Phi$ touches $u^{(k)}$ by below at $(x_0,t_0)$. Hence $\Phi_y (x_0,0,t_0) \leq g^{(k)}(x_0,t_0)$. On the other hand,  $\Phi_y (x_0,0,t_0)=b_y(0,0)$ which completes the proof.
\end{proof}

\begin{prop} \label{u^k_convergence}
$u^{(k)} \to u$ uniformly in $\overline{Q}_1^+$.
\end{prop}

\begin{proof}
We split our proof into two steps:
\\ \textbf{Step 1.} We prove equicontinuity of $u^{(k)}$. For, it is enough to obtain an independent of $k$ modulus of continuity of $u^{(k)}$ in $\overline{Q}_1^+$. Note that Lemma \ref{u^k_est} gives a uniform $L^\infty$-bound for $u^{(k)}$ in $\overline{Q}_1^+$. Also Lemma \ref{g^k_est} gives a uniform $L^\infty$-bound for $g^{(k)}$, thus using Theorem 6 in \cite{CM19} we get a uniform $H^\alpha$-estimate for $u^{(k)}$ in $\overline{Q}^+_{1-\frac{\rho}{2}}$. So it remains to get a uniform modulus of continuity in $\overline{Q}_1^+ \setminus \overline{Q}^+_{1-\frac{\rho}{2}}$.

Note that $\left(u^{(k)}\right)_y$=0 on $Q_1^* \setminus Q^*_{1-\rho}$. Thus if we extend $u^{(k)}$ in $Q_1 \setminus Q_{1-\rho}$ considering its even reflection $\tilde{u}^{(k)}$ with respect to $y$ we have that $\tilde{u}^{(k)} \in S_p(\lambda, \Lambda)$ (see Proposition \ref{reflection}). We observe also that $\tilde{u}^{(k)}|_{\partial_pQ_1}=u_0$ is independent of $k$ and smooth enough and $\tilde{u}^{(k)}|_{\partial_pQ_{1-\rho}}$ satisfy uniform $H^\alpha$-estimate. Thus using global $H^\alpha$-estimates for Dirichlet problems we get the desired uniform modulus in $\overline{Q}_1^+ \setminus \overline{Q}^+_{1-\frac{\rho}{2}}$. 
\\ \textbf{Step 2.} Arzel\'a-Ascoli lemma implies that every subsequence of $\{u^{(k)}\}$ has a subsequence that converges uniformly in $\overline{Q}^+_1$. We claim that every uniformly convergent subsequence of $\{u^{(k)}\}$ must converge to $u$, then we should have that $u^{(k)} \to u$ uniformly in $\overline{Q}_1^+$. To prove this claim let $v$ be the uniform limit of $\{u^{(k_m)}\}$ in $\overline{Q}_1^+$. If we show that $v$ satisfies problem (\ref{thin_prob}) then $v=u$ by uniqueness. The closedness result of Proposition 31 in \cite{CM19} gives immediately that $F(D^2v)-v_t=0$ in $Q_1^+$ and $v_y \leq 0$ on $Q_1^*$ in the viscosity sense. Additionally, $v=u_0$ on $\partial_p Q_1^+ \setminus Q_1^*$. It remains to check that
\begin{enumerate}
\item $v_y=0$ on $Q_1^* \cap \{v > \varphi \}$, in the viscosity sense.
\item $v \geq \varphi$ on $Q_1^*$.
\end{enumerate}

For $(1)$ let $(x_0,t_0) \in Q_1^*$ be so that $v(x_0,0,t_0) > \varphi (x_0,t_0)$. From the continuity of $v$ and $\varphi$, there exists some small $\delta>0$ so that $v(x,0,t) > \varphi(x,t)$ for any $(x,t) \in \overline{Q}_\delta^*(x_0,t_0)$. Next we use the uniform convergence of $u^{(k_m)}$ to $v$. Take $\varepsilon := \min_{\overline{Q}_\delta^*} (v-\varphi)>0$
then there exists $n_0 \in \N$ so that $|u^{(k_m)}-v|<\varepsilon$ in $\overline{Q}_\delta^*(x_0,t_0)$ for any $m \geq n_0$. Hence $u^{(k_m)}-v>-\varepsilon \geq -v +\varphi$, that is $u^{(k_m)}>\varphi$, so $\left(u^{(k_m)}\right)_y=0$ in $\overline{Q}_\delta^*(x_0,t_0)$ for any $m \geq n_0$. Since $F\left( D^2u^{(k_m)}\right)-\left( u^{(k_m)} \right)
_t=0$ in $\overline{Q}_\delta^+(x_0,t_0)$ again from the closedness result of Proposition 31 in \cite{CM19} we get that $v_y=0$ on $\overline{Q}_\delta^*(x_0,t_0)$.

For $(2)$ we assume that there exists some $(x_0,t_0) \in Q_1^*$ such that $v(x_0,0,t_0) < \varphi(x_0,t_0)$ to get a contradiction. Again using the convergence we have that there exists $n_0 \in \N$ so that $u^{(k_m)}(x_0,0,t_0)-v(x_0,0,t_0) < \varphi(x_0,t_0)-v(x_0,0,t_0)$ for any $m \geq n_0$. Hence $g^{(k_m)}(x_0,0,t_0) =-k_m \left( \varphi(x_0,t_0) - u^{(k_m)}(x_0,0,t_0)\right)$, that is, $\varphi(x_0,t_0) - u^{(k_m)}(x_0,0,t_0) = -\frac{1}{k_m}g^{(k_m)}(x_0,0,t_0)$ for any $m \geq n_0$ and $g^{(k)}$ is bounded independently of $k$ by Lemma \ref{g^k_est}. By taking $m \to \infty$ we get that $\varphi(x_0,t_0) = v(x_0,0,t_0)$ which is a contradiction. 
\end{proof}

Proposition \ref{u^k_convergence} gives the following.

\begin{lemma} \label{Du^k_convergence}
For any $0<\delta<1$, $Du^{(k)} \to Du$ uniformly in $K_\delta:=Q_{1-\delta} \cap \{y > \delta \}$.
\end{lemma}

\begin{proof}
Note first that from interior $H^{1+\alpha}$-estimates for viscosity solutions of $F(D^2v)-v_t=0$ we know the existence of   $Du^{(k)}, Du$ in $K_\delta$ and a uniform $H^\alpha$-estimate for $Du^{(k)}$ (recall that $||u^{(k)}||_{L^\infty\left(Q_1^+\right)}$ are uniformly bounded). Therefore using Arzel\'a-Ascoli lemma we get that every subsequence of $\{Du^{(k)}\}$ has a subsequence that converges uniformly in $\overline{K}_\delta$. Then by standard calculus we know that any uniformly convergent subsequence of $\{Du^{(k)}\}$  should converge to $Du$.
\end{proof}

\begin{lemma} \label{u^k_regularity}
For any $0<\delta<1$, $u^{(k)} \in H^{1+\alpha}\left( \overline{Q}^+_{1-\delta}\right)$.
\end{lemma}

Although the $H^{1+\alpha}$-estimates of the above may depend on $k$, Lemma \ref{u^k_regularity} ensures the existence and regularity of $\left(u^{(k)}\right)_y$ on $Q_1^*$ in the classical sense.

\begin{proof}
Using Lemma \ref{g^k_est} and Theorem 6 in \cite{CM19} we get a uniform $H^\alpha$-estimate for $u^{(k)}$ in $\overline{Q}^+_{1-\frac{\delta}{2}}$ which means that $g^{(k)}=-k \left(\varphi-u^{(k)} \right)^+$ is $H^\alpha$ on $\overline{Q}^*_{1-\frac{\delta}{2}}$. Then applying Theorem 17 in \cite{CM19} we get the desired.
\end{proof}

Now we proceed in proving (\ref{sigmasign1}).

\begin{lemma} \label{sigmasign}
$\sigma \leq 0$ on $Q_1^*$.
\end{lemma}

\begin{proof}
For $k \in \N$ (fixed), we consider the solution  $u^{(k)}$ of (\ref{pen_prob}). We denote by $v:=\left(u^{(k)}\right)_y$ which exists in the classical sense and it is continuous in $Q_1^+ \cup Q_1^*$ (due to Lemma \ref{u^k_regularity}). Then $v \leq 0$ on $Q_1^*$ and if $0<\delta<\rho$, then $v=0$ on $Q_1^* \setminus Q_{1-\delta}^*$. Moreover we can use Theorem 15 of \cite{CM19} in $Q_{1-\frac{\delta}{3}}^+ \setminus Q_{1-\frac{2\delta}{3}}^+$ to obtain that 
$$v\leq M, \ \ \text{ on } \ \ \partial_p Q_{1-\frac{\delta}{2}}^+ \setminus Q_{1-\frac{\delta}{2}}^* $$
where $M>0$ is a constant independent of $k$.

Next we apply a barrier argument to $v$. We define the function $b$ to be the viscosity solution of 
\begin{align*}
\begin{cases}
\mathcal{M}^+\left(D^2b, \frac{\lambda}{n}, \Lambda \right) -b_t= 0, &\ \text{ in } \ \ Q^+_{1-\frac{\delta}{2}} \\
b=M, &\ \text{ on } \ \ \partial_pQ^+_{1-\frac{\delta}{2}}\setminus Q_{1-\frac{\delta}{2}}^* \\
b=0, &\ \text{ on } \ \ Q^*_{1-\delta} \\
b(x,0,t) = \frac{2M}{\delta} \left( \max\left\lbrace|x|,|t|^{\frac{1}{2}}\right\rbrace-1+\delta\right), &\ \text{ on } \ \ Q_{1-\frac{\delta}{2}}^* \setminus Q^*_{1-\delta}.
\end{cases}
\end{align*}
We remark that $v\leq b$ on $\partial_p Q_{1-\frac{\delta}{2}}^+ \setminus Q_{1-\frac{\delta}{2}}^* \ $ and $\ v\leq 0 \leq b$ on $Q_{1-\frac{\delta}{2}}^*$. Finally we know that $v \in \underline{S}_p\left( \frac{\lambda}{n}, \Lambda\right)$ in $Q_{1-\frac{\delta}{2}}^+ $, then $v-b \in \underline{S}_p\left( \frac{\lambda}{n}, \Lambda\right)$ in $Q_{1-\frac{\delta}{2}}^+ $. Using maximum principle we get that $v\leq b$ in $Q_{1-\frac{\delta}{2}}^+$ and note that function $b$ does not depend on $k$. On the other hand $\left(u^{(k)}\right)_y \to u_y$ as $k \to \infty$ pointwise in $Q_{1-\frac{\delta}{2}}^+$ by Lemma \ref{Du^k_convergence}. Hence $u_y \leq b$ in $Q_{1-\frac{\delta}{2}}^+$. Finally, we observe that $b=0$ on $Q_{1-\delta}^*$, for any $0<\delta<\rho$ and we take $y \to 0^+$.
\end{proof}

\section{Regularity of the solution}

As we have mentioned at the points of $\Omega^*$ the regularity is known, therefore at these points the viscosity Neumann condition holds in the classical sense, thus $\sigma=0$ in $\Omega^*$. 

In this section we concentrate in studying the regularity of $\sigma$ around free boundary points in order to treat our problem as a non-homogeneous Neumann boundary value problem around these points. To achieve this we show first Lemma \ref{sigmaregularity}, an $H^\alpha$-estimate for $\sigma$ in universal neighborhoods of points of $\Omega^*$. Lemma \ref{sigmaregularity} is based on Lemmata \ref{measure} and \ref{Harnack_type} and on semi-concavity of $u$ in $y$. Lemma \ref{measure} says that considering a non-contact point $P_0 \in Q_{1/2}^*$, we can find a universal neighborhood of $P_0$ which contains a small universal thin-cylinder where $\sigma$ decays proportionally to its radius. Finally Lemma \ref{Harnack_type} says that the information we have inside this small thin cylinder can be carried to a suitable set inside $Q_1^+$ and then is carried back in a parabolic neighborhood of $P_0$ using semi-concavity in $y$. An iterative application of the above gives Lemma \ref{sigmaregularity}.

We start with Lemma \ref{comparetophi} which is important in proving Lemma \ref{measure}. The following  simple remark is useful.
 
\begin{preremark} \label{comparetophiremark}
For $P_0:=(x_0,t_0) \in \Omega^*$, $K_0:=2K$  and 
$$\tilde{\varphi}_{P_0}(x,t):= \varphi (x_0,t_0)+D\varphi (x_0,t_0) \cdot (x-x_0)-K_0(t-t_0) + K_0|x-x_0|^2.$$
we have that $\tilde{\varphi}_{P_0}>\varphi$ in $Q_1^*\cap \{t \leq t_0 \} \setminus \{(x_0,t_0)\}$.
\end{preremark}

Indeed, let $\Phi= \tilde{\varphi}_{P_0}-\varphi$. Then we observe that $\Phi (x_0,t_0)=0$ and
\begin{enumerate}
\item[\textit{(a)}] $D\Phi(x,t)=D\varphi(x_0,t_0)+2K_0(x-x_0)-D\varphi(x,t)$, thus $D\Phi(x_0,t_0)=0$.
\item[\textit{(b)}] $D^2\Phi(x,t)=2K_0I_{n-1}-D^2\varphi(x,t)>0$, that is $\Phi$ is convex with respect to $x$.
\item[\textit{(c)}] $\Phi_t(x,t)=-2K_0-\varphi_t(x,t)<0$, that is $\Phi$ is monotone decreasing with respect to $t$.
\end{enumerate}
Then \textit{(b)} (through integration) gives that $\Phi(x,t_0)- \Phi(x_0,t_0)> (x-x_0) \cdot D\Phi(x_0,t_0)=0$ for $x \neq x_0$. Thus by \textit{(a)} we have that $\Phi(x,t_0)>  \Phi(x_0,t_0)  =0$, for $x \neq x_0$. On the other hand \textit{(c)} gives that $\Phi(x,t)> \Phi(x,t_0)$ for any $t<t_0$ and any $x$. Combining the above we get that $\Phi(x,t)>  \Phi(x_0,t_0)  =0$, for any $x \neq x_0$ and any $t<t_0$.

\begin{lemma} \label{comparetophi}
For $P_0=(x_0,t_0) \in \Omega^*$, $K_0:=2K$ and $C_0 > \frac{n}{\lambda} \left[ \Lambda (n-1)+1 \right]$ we define 
$$h_{P_0}(x,y,t):= \varphi (x_0,t_0)+D\varphi (x_0,t_0) \cdot (x-x_0)-K_0(t-t_0) + K_0|x-x_0|^2 -C_0 K_0 y^2.$$
We consider any set of the form $\Theta := \tilde{\Theta}\times (t_1,t_0] \subset Q_1$, with $P_0 \in \Theta$, $\tilde{\Theta} \subset \R ^n$ a bounded domain containing $x_0$ and $0<t_1<t_0$. Then
$$\sup_{\partial_p\Theta \cap \{y \geq 0\}}(u-h_{P_0}) \geq 0.$$
\end{lemma}

\begin{proof}
Let $w:=u-h_{p_0}$ then we have that $w(x_0,0,t_0)=u(x_0,0,t_0)-\varphi(x_0,t_0)>0$, since $(x_0,t_0) \in \Omega^*$. Moreover, $w \in \underline{S}_p\left(\frac{\lambda}{n}, \Lambda \right)$ in $Q_1^+$. Indeed, we note that $\left(h_{P_0}\right)_{ij}=0$ for $i\neq j$, $\left(h_{P_0}\right)_{ii}=2K_0$ for $i <n$, $\left(h_{P_0}\right)_{nn}=-2C_0K_0$ and $\left(h_{P_0}\right)_{t}=-K_0$. Then $\mathcal{M}^+\left(D^2h_{P_0}, \frac{\lambda}{n}, \Lambda \right) -\left(h_{P_0}\right)_{t}<-K_0<0 $ in the classical sense which gives the desired. Finally, $w_y=0$ on $\Omega^*$ in the classical sense. Indeed, it is enough to note that $\left(h_{P_0}\right)_{y}=-\frac{2K_0n^2\Lambda}{\lambda} \ y$, that is $\left(h_{P_0}\right)_{y}=0$ on $Q_1^*$.

Now we denote by $w^*$ the extension of $w$ in $Q_1$ considering its even reflection with respect to $y$ and we have that $w^* \in \underline{S}_p\left(\frac{\lambda}{n}, \Lambda \right)$ in $Q_1 \setminus \Delta^*$ (see Proposition \ref{reflection}). Then maximum principle gives that
$$\sup_{\partial_p \left(\Theta \setminus \Delta^* \right)\cap \{y\geq 0\}} w = \sup_{\partial_p \left(\Theta \setminus \Delta^* \right)} w^*\geq \sup_{\Theta \setminus \Delta^*} w^*\geq w(x_0,0,t_0)>0$$
since $(x_0,t_0) \in \Theta \setminus \Delta^* $. Finally we observe that $\ \partial_p \left(\Theta \setminus \Delta^* \right)\cap \{y\geq 0\} \subset \left( \partial_p \Theta \cap \{y\geq 0\}  \right) \cup \left( \Delta^* \cap \{t\leq t_0\} \right)$. On the other hand, $h_{P_0}=\tilde{\varphi}_{P_0}>\varphi$ on $Q_1^*\cap \{t \leq t_0 \} \setminus \{(x_0,t_0)\}$ from Remark \ref{comparetophiremark} and $\varphi=u$ on $\Delta^*$, that is $w<0$ on $\Delta^* \cap \{t\leq t_0\}$ and the proof is complete. 
\end{proof}

\begin{lemma} \label{measure}
For $\gamma>0$ we define $\Omega_\gamma^* := \{(x,t) \in Q_1^* : \sigma (x,t) > -\gamma\}$. Let $(x_0,t_0) \in \Omega^* \cap Q_{1/2}^*$, then there exist constants $0<\bar{C}<\bar{\bar{C}}<1$ which depend only on $K$, $n, \lambda , \Lambda,\rho$ so that for any $0<\gamma< \frac{1}{2}$ there exists a thin-cylinder $Q^*_{\bar{C}\gamma}(\bar{x},\bar{t})$ so that
$$Q^*_{\bar{C}\gamma}(\bar{x},\bar{t}) \subset Q^*_{\bar{\bar{C}}\gamma}(x_0,t_0) \cap \Omega_\gamma^*.$$
\end{lemma}

\begin{proof}
Let $(x_0,t_0) \in \Omega^* \cap Q_{1/2}^*$, we apply Lemma \ref{comparetophi} with 
$$\Theta:=B^*_{C_1\gamma}(x_0) \times (-C_2\gamma, C_2 \gamma) \times \left( t_0-(C_1\gamma )^2,t_0\right]$$
where $0<C_2<<C_1<<1$ to be chosen. Then there exists $P_1=(x_1,y_1,t_1) \in \partial_p\Theta \cap \{y \geq 0\}$ so that 
\begin{equation} \label{measure1}
u(P_1)-h_{P_0}(P_1) \geq 0.
\end{equation}

We split into two cases.
\\ \underline{\textit{Case 1.}} If $\ |x_1-x_0|=C_1\gamma \ $ or $\ t_1=t_0-(C_1\gamma )^2$. Then using (\ref{measure1}) and Remark \ref{comparetophiremark} we have in the first occasion that
\begin{align*}
u(P_1)&\geq \varphi (x_0,t_0)+D\varphi (x_0,t_0) \cdot (x_1-x_0)-K_0(t_1-t_0) + \frac{K}{2}|x_1-x_0|^2 \\
&\ \ \ +\frac{K}{2}|x_1-x_0|^2 - \frac{Kn^2\Lambda}{\lambda}y_1^2 \\
&\geq \varphi(x_1,t_1) +\frac{K}{2}(C_1\gamma )^2 - \frac{Kn^2\Lambda}{\lambda}(C_2\gamma )^2
\end{align*}
and similarly in the second occasion that $u(P_1) \geq \varphi(x_1,t_1) +\frac{K_0}{2}(C_1\gamma )^2 - \frac{K_0n^2\Lambda}{\lambda}(C_2\gamma )^2$.

Thus in any case
\begin{equation} \label{measure2}
u(x_1,y_1,t_1)\geq \varphi(x_1,t_1)+C_4\gamma^2 
\end{equation}
where $C_4>0$ a constant depending only on universal constants and on $C_1, C_2$ (choosing $0<C_2 < \sqrt{\frac{\lambda}{2n^2\Lambda}} \ C_1$).

Now take any $(x_2,t_2) \in Q^*_{C_3 \gamma}(x_1,t_1)$, for $C_3$ to be chosen. We intend to transfer the information (\ref{measure2}) from $(x_1,y_1,t_1)$ to $(x_2,t_2)$ through integration and using the bounds of Proposition \ref{semiconcavity} for suitable derivatives. We denote by $\tau=\frac{x_2-x_1}{|x_2-x_1|} \in \R^{n-1}$ and we assume that $(x_2-x_1) \cdot D_{n-1}(u-\varphi)(P_1) \geq 0$ (considering the extension of $\varphi$ in $Q_1^+$ where $\varphi^*(x,y,t)=\varphi(x,y)$). We notice that
\begin{align*}
&\int_{0}^{|x_2-x_1|} \int_0^e (u-\varphi)_{\tau \tau}(x_1+\tau h,y_1,t_1) \ dh de \\
&\ \ \ \ \ \ \ \ = (u-\varphi)(x_2,y_1,t_1)-(u-\varphi)(x_1,y_1,t_1) - |x_2-x_1| (u-\varphi)_{\tau}(x_1,y_1,t_1)
\end{align*}
and
\begin{align*}
&\int_{t_2}^{t_1}  (u-\varphi)_{t}(x_2,y_1,h) \ dh  = (u-\varphi)(x_2,y_1,t_1)-(u-\varphi)(x_2,y_1,t_2).
\end{align*}
Combining the above we get
\begin{align} \label{measure3}
&\int_{0}^{|x_2-x_1|} \int_0^e (u-\varphi)_{\tau \tau}(x_1+\tau h,y_1,t_1) \ dh de - \int_{t_2}^{t_1}  (u-\varphi)_{t}(x_2,y_1,h) \ dh \nonumber \\
&\ \ \ = (u-\varphi)(x_2,y_1,t_2)-(u-\varphi)(x_1,y_1,t_1) - |x_2-x_1| (u-\varphi)_{\tau}(x_1,y_1,t_1).
\end{align}
On the other hand using \textit{(B)} of Proposition \ref{semiconcavity} we have
\begin{align*}
\int_{0}^{|x_2-x_1|} \int_0^e (u-\varphi)_{\tau \tau}(x_1+\tau h,y_1,t_1) \ dh de \geq -C |x_2-x_1|^2 \geq -C (C_3\gamma)^2
\end{align*}
and
\begin{align*}
- \int_{t_2}^{t_1}  (u-\varphi)_{t}(x_2,y_1,h) \ dh \geq -C (t_1-t_2) \geq -C (C_3\gamma)^2.
\end{align*}
Therefore returning to (\ref{measure3}) we have that
\begin{align} \label{measure4}
(u-\varphi)(x_2,y_1,t_2) &\geq (u-\varphi)(x_1,y_1,t_1) + (x_2-x_1) \cdot D_{n-1}(u-\varphi)(x_1,y_1,t_1) -C (C_3\gamma)^2 \nonumber \\
&\geq C_4\gamma^2-C (C_3\gamma)^2>0
\end{align}
by choosing $0<C_3^2<\frac{C_4}{C}$.

Now (to get a contradiction) we assume that $(x_2,t_2) \notin \Omega_\gamma^*$, that is $\sigma(x_2,t_2)\leq -\gamma<0$. Then $(x_2,t_2) \in \Delta^*$, that is $u(x_2,0,t_2)=\varphi(x_2,t_2)$. Similarly as before we want to transfer this information from $(x_2,0,t_2)$ to $(x_2,y_1,t_2)$ via integration of $u_{yy}$ and using \textit{(C)} of Proposition \ref{semiconcavity}. We have
\begin{align*}
Cy_1^2 \geq \int_{0}^{y_1} \int_0^e u_{yy}(x_2,h,t_2) \ dh de=u(x_2,y_1,t_2)-u(x_2,0,t_2)-y_1 \sigma(x_2,t_2)
\end{align*}
then, $u(x_2,y_1,t_2)-\varphi(x_2,t_2)\leq Cy_1^2+y_1(-\gamma)\leq y_1 \gamma (CC_2-1)<0$, choosing $0<C_2\leq \frac{1}{C}$. This is a contradiction regarding (\ref{measure4}).

\hspace{-7mm} \underline{\textit{Case 2.}} If $\ y_1=C_2\gamma  $. Then using (\ref{measure1}) and Remark \ref{comparetophiremark} we have 
\begin{equation} \label{measure5}
u(x_1,y_1,t_1)\geq \varphi(x_1,t_1) -\frac{K_0n^2\Lambda}{\lambda}C_2^2\gamma^2. 
\end{equation}

We take any $(x_2,t_2) \in Q^*_{C_2 \gamma}(x_1,t_1)$. Assuming that $(x_2-x_1) \cdot D_{n-1}(u-\varphi)(P_1) \geq 0$ we can repeat the computations of Case 1 slightly modified to obtain
\begin{align} \label{measure6}
(u-\varphi)(x_2,C_2\gamma,t_2) \geq -CC_2^2\gamma^2>-C_6C_2 \gamma^2
\end{align}
where $0<C_6<CC_2$.

Now (to get a contradiction) we assume that $\sigma(x_2,t_2)\leq -\gamma<0$. Then $u(x_2,0,t_2)=\varphi(x_2,t_2)$. Similarly as in Case 1 we get that $u(x_2,C_2\gamma,t_2)-\varphi(x_2,t_2)\leq C_2\gamma^2(CC_2-1)<-C_6C_2 \gamma^2$, choosing $0<C_6<1-CC_2$ and $C_2$ small enough. This is a contradiction regarding (\ref{measure6}).

In any case we have that there exists $0<C_7<<1$ depending only on $\rho, n, \lambda , \Lambda$, $K$ so that if $(x_2,t_2) \in Q^*_{C_7 \gamma}(x_1,t_1)$ with $(x_2-x_1) \cdot D_{n-1}(u-\varphi)(x_1,y_1,t_1) \geq 0$ (which roughly speaking holds at least in the "half" of $Q^*_{C_7 \gamma}(x_1,t_1)$)  then $(x_2,t_2) \in \Omega^*_\gamma$. Moreover choosing $1>\bar{\bar{C}}>C_7+C_1$ it is easy to check that $Q^*_{C_7 \gamma}(x_1,t_1) \subset Q^*_{\bar{\bar{C}}\gamma}(x_0,t_0)$. By choosing a thin cylinder $Q^*_{\bar{C}\gamma}(\bar{x},\bar{t})$ inside $Q^*_{C_7 \gamma}(x_1,t_1) \cap \{(x_2-x_1) \cdot D_{n-1}(u-\varphi)(x_1,y_1,t_1) \geq 0\}$ the proof is complete.
\end{proof}

Now maximum principle and a barrier argument give the following important property.

\begin{lemma} \label{Harnack_type}
Consider the set $K_1:=B^*_1 \times (0,1) \times (-1,0]$ and assume that $w \in C\left(K_1\right)$ satisfies in the viscosity sense
\begin{align*}.
\begin{cases}
\mathcal{M}^- \left( D^2w,\lambda, \Lambda \right)-w_t \leq 0, &\ \ \text{ in } \ \ K_1 \\
w \geq 0, &\ \ \text{ in } \ \ K_1. \end{cases}
\end{align*}
Suppose that there exists some neighborhood $Q^*_\delta(\bar{x},\bar{t}) \subset Q_1^*$ so that
$$\liminf_{y\to 0^+}w(x,y,t) \geq 1, \ \ \text{ for any } \ \ (x,t) \in  \overline{Q}^*_\delta(\bar{x},\bar{t}).$$
Then, there exists $\varepsilon=\varepsilon (\delta,n,\lambda ,\Lambda)>0$ so that
$$w(x,y,t) \geq \varepsilon, \ \ \text{ for any } \ \ (x,y,t) \in \overline{B}^*_{1/2}\times \left[\frac{1}{4},\frac{3}{4}\right]\times \left[-\frac{\delta^2}{2},0\right].$$
\end{lemma}

\begin{proof}
For any $P'=(x',t') \in \overline{Q}_{1-\delta}^*$ we define the auxiliary function
\begin{align*}
\begin{cases}
\mathcal{M}^-\left(D^2b_{P'}, \frac{\lambda}{n}, \Lambda \right) -\left(b_{P'}\right)_t= 0, &\ \text{ in } \ \ K_1 \\
b_{P'}=0, &\ \text{ on } \ \ \partial_p K_1 \setminus Q^*_{\delta}(P') \\
b_{P'}(x,0,t) = 1-\frac{1}{\delta}  \max\{|x-x'|,\sqrt{2}|t-t''|^{\frac{1}{2}}\}, &\ \text{ on } \ \  Q^*_{\delta}(P')
\end{cases}
\end{align*}
where $t'':=t'-\frac{\delta^2}{2}$. Applying regularity results for Dirichlet-type boundary value problems (see \cite{Wang2}) we have that $b_{P'}$ is Lipschitz in $\overline{K}_1$ with the corresponding constant depending only on $\delta$ and universal quantities (but not on $P'$).

We claim that 
\begin{equation} \label{Harnack_type1}
b_{P'} > 0, \ \ \text{ in } \ \ \overline{B}^*_{1/2}\times \left[\frac{1}{4},\frac{3}{4}\right]\times \left[-\frac{\delta^2}{2},0\right]:=K_2.
\end{equation}
Indeed, note first that $b_{P'} \geq 0$ on $\partial_pK_1$, thus by maximum principle $b_{P'} \geq 0$ in $\overline{K}_1$. We suppose that there exists some $(x_1,y_1,t_1) \in K_2$ with $b_{P'}(x_1,y_1,t_1)=0$ which means that $b_{P'}$ attains its minimum over $\overline{K}_1$ at $(x_1,y_1,t_1)$. Then strong maximum principle gives that $b_{P'} = 0$ on $\overline{K}_1 \cap \{t \leq t_1\}.$ Note that $t_1 \geq -\frac{\delta^2}{2} >t'-\delta^2$ then there exists $(x,t) \in$ $Q^*_{\delta}(P')$ such that $t<t_1$, that is $b_{P'}(x,0,t) >0$ and $t<t_1$ which is a contradiction.

Now let $\ \varepsilon (P',\delta,n,\lambda,\Lambda):= \min_{K_2} b_{P'}>0 \ $ and 
$$\tilde{\varepsilon} (\delta,n,\lambda,\Lambda):=\inf_{P' \in Q^*_{1-\delta}} \varepsilon (P',\delta,n,\lambda,\Lambda) \geq 0.$$
We want to show that $\tilde{\varepsilon}>0$. We assume that $\tilde{\varepsilon}=0$, then there exists $\{P'_j:=(x'_j,t'_j)\}_{j \in \N} \subset Q^*_{1-\delta}$ so that $\varepsilon (P'_j,\delta,n,\lambda,\Lambda) \to 0$ as $j \to \infty$. Also for any $j \in \N$ there exists $(X_j,t_j) \in K_2$ so that $\varepsilon (P'_j,\delta,n,\lambda,\Lambda)=b_{P'_j}(X_j,t_j)$. We notice also that $\{P'_j\}, \ \{ (X_j,t_j) \}$ are both bounded sequences and therefore there exist convergent subsequences (for which we use the same indices for simplicity). That is
$$P'_j \to P_\infty ' \in \overline{Q}_{1-\delta}^*, \ \ (X_j,t_j) \to (X_\infty,t_\infty), \ \ \text{ as } \ j \to \infty.$$

On the other hand $b_{P'_j}$ are equicontinuous and uniformly bounded in $\overline{K}_1$, thus there exist a uniformly convergent subsequence in $\overline{K}_1$, that is $b_{P'_j} \to b_{\infty}$ uniformly in $\overline{K}_1$ as $j \to \infty$. To get the contradiction it is enough to show that
\begin{equation} \label{Harnack_type2}
b_\infty=b_{P'_\infty}, \ \ \text{ in } \ \ \overline{K}_1.
\end{equation} 
Indeed, if (\ref{Harnack_type2}) holds then by uniform convergence we have that $b_{P'_j}(X_j,t_j) \to b_{P'_\infty}(X_\infty,t_\infty)$ as $j \to \infty$, thus $b_{P'_\infty}(X_\infty,t_\infty)=0$ which contradicts (\ref{Harnack_type1}) since $(X_\infty,t_\infty) \in K_2$. Now to obtain (\ref{Harnack_type2}), using uniqueness, it is enough to show that $b_\infty$ solves the same Dirichlet problem as $b_{P'_\infty}$ in $\overline{K}_1$. From closedness of viscosity solutions we know that $\mathcal{M}^-\left(D^2b_{\infty}, \frac{\lambda}{n}, \Lambda \right) -\left(b_{\infty}\right)_t= 0$ in $K_1$. Also $b_{\infty}=0$ on $\partial_p K_1 \setminus Q^*_1$. Thus it remains to check the following two
\begin{enumerate}
\item $b_\infty(x,0,t) = 1-\frac{1}{\delta}  \max\{|x-x'_\infty|,\sqrt{2}|t-t''_\infty|^{\frac{1}{2}}\} \ $ on $\ \overline{Q}^*_{\delta}(P_\infty')$
\item $b_{\infty}=0\ $ on $\ Q^*_1 \setminus \overline{Q}^*_{\delta}(P'_\infty)$.
\end{enumerate}
For $(x,t)$ such that $|x-x_\infty'|<\delta$ and $|t-t_\infty''|<\frac{\delta^2}{2}$ we can choose an integer $m=m(x,t,\delta) > \frac{3}{2\delta}$ so that $(x,t) \in Q^*_{\delta-\frac{1}{m}} \left(  P_\infty' \right)$. Also there exists integer $N=N(x,t,\delta) \in \N$ so that for any $j \geq N$, $|x_j'-x_\infty'|<\frac{1}{m}$ and $|t_j''-t_\infty''|<\frac{1}{m^2}.$ Then for any $j \geq N$, using that $m>\frac{3}{2\delta}$, we have that $(x,t) \in \overline{Q}^*_{\delta}(P_j')$, that is $b_{P_j'}(x,0,t) = 1-\frac{1}{\delta}  \max\{|x-x'_j|,\sqrt{2}|t-t''_j|^{\frac{1}{2}}\}$ and taking $j \to \infty$ we obtain (1) at $(x,t)$. Note that for $(x,t)$ such that $|x-x_\infty'|=\delta$ or $t_\infty'-\delta^2=t$ or $t=t_\infty'$ we use the continuity of $b_\infty$. Finally for $(x,t) \in  Q^*_1 \setminus \overline{Q}^*_{\delta}(P'_\infty)$,  we follow a similar argument as before by choosing $m=m(x,t,\delta) > \frac{1}{\delta}$ so that $(x,t) \in Q^*_1 \setminus \overline{Q}^*_{\delta+\frac{1}{m}} \left(  P_\infty' \right)$.  

Now if $\bar{P}=(\bar{x},\bar{t})$ the given point we have that $b_{\bar{P}}\geq \tilde{\varepsilon}$ in $K_2$. We use maximum principle to get this bound for $w$ as well. So let $v=w-b_{\bar{P}}$ then $v \in \overline{S}_p\left(\frac{\lambda}{n},\Lambda\right)$ in $K_1$. Moreover if $(x,t) \in Q^*_\delta(\bar{P})$ then $\liminf_{y \to 0^+} v(x,y,t) \geq 1-b_{\bar{P}}(x,0,t)\geq 0$ and if $(x,t) \in \partial_pK_1 \setminus Q^*_\delta(\bar{P})$ then $\liminf_{y \to 0^+} v(x,y,t) \geq  0$ since $w \geq 0$. Therefore, $w \geq b_{\bar{P}} \geq \tilde{\varepsilon}$ in $K_2$.
\end{proof}

The next lemma is a consequence of an iterative argument.

\begin{lemma}\label{sigmaregularity}
Let $(x_0,t_0) \in \Omega^* \cap Q_{1/2}^*$, then there exists  universal constants $0<\alpha<1$, $C>0$ so that
$$0\geq \sigma (x,t) \geq -C \left( |x-x_0|+|t-t_0|^{1/2}\right)^\alpha, \ \ \text{ for any } \ (x,t) \in Q_{1/2}^*(x_0,t_0).$$
\end{lemma}

\begin{proof}
Our aim is to show by induction that for any $k \in \N$
\begin{equation} \label{sigmaregularity1}
u_y(X,t) \geq - C \theta^k, \ \ \text{ for every } \ \ (X,t) \in Q^*_{r^k}(x_0,t_0) \times \{y \in (0,r^k)\}
\end{equation}
where $0<r <<\theta<1$ to be chosen and $C>0$ universal.  We proceed by induction. For $k=1$ it follows by \textit{(A)} of Proposition \ref{semiconcavity} by choosing $C$ appropriately. We assume that (\ref{sigmaregularity1}) holds for some $k$ and we prove it for $k+1$.

We define 
$$w=\frac{u_y+C\theta^k}{-\mu r^k+C\theta^k}, \ \ \text{ in } \ \ Q^*_{r^k}(x_0,t_0) \times \{y \in (0,r^k)\}$$
where $0<\mu<1$ a small constant to be chosen. Then by the hypothesis of the induction and choosing $r<\theta$ and $\mu<C$ we have that $w \geq 0$ in $Q^*_{r^k}(x_0,t_0) \times \{y \in (0,r^k)\}$. Moreover, $\mathcal{M}^-\left(D^2w,\frac{\lambda}{n},\Lambda\right)-w_t \leq 0$ in $Q^*_{r^k}(x_0,t_0) \times \{y \in (0,r^k)\}$. We observe also that 
$$\lim_{y \to 0^+} w(x,y,t)=\frac{\sigma(x,t)+C\theta^k}{-\mu r^k+C\theta^k}, \ \ \text{ for } \ \ (x,t) \in Q^*_{r^k}(x_0,t_0).$$

On the other hand applying Lemma \ref{measure} around $(x_0,t_0) \in \Omega^* \cap Q^*_{1/2}$ with $\gamma = \mu r^k<\mu r<\frac{1}{2}$ we get that there exists $Q^*_{\bar{C} \mu r^k}(\bar{x},\bar{t}) \subset Q^*_{ \mu r^k}(x_0,t_0) \cap \Omega_{ \mu r^k}^*$, where $0<\bar{C}<1$ depends only on $K$, $n, \lambda , \Lambda$ and $\rho$. That is, $\lim_{y \to 0^+} w(x,y,t)\geq 1$ for $(x,t) \in Q^*_{\bar{C} \mu r^k}(\bar{x},\bar{t}).$ Therefore, $w$ satisfies the assumptions of Lemma \ref{Harnack_type} in $Q^*_{r^k}(x_0,t_0) \times (0,r^k)$. So we apply Lemma \ref{Harnack_type} to the rescaled function $W(x,y,t):=w(\mu r^kx+x_0, \mu r^k y, (\mu r^k)^2t+t_0)$ in $K_1$ and obtain that
\begin{equation} \label{sigmaregularity2}
w \geq \varepsilon, \ \ \text{ in } \ \overline{B}^*_{\frac{\mu r^k}{2}}(x_0) \times \left[\frac{\mu r^k}{4}, \frac{3\mu r^k}{4}\right] \times \left[t_0 - \frac{(\bar{C} \mu r^k)^2}{2}, t_0 \right]
\end{equation}
where $\varepsilon=\varepsilon (\bar{C},n,\lambda,\Lambda)>0$, that is, $u_y \geq -C\theta^k+\frac{\varepsilon C \theta^k}{2}$ using that $r<\theta$ and choosing $\mu < \frac{C}{2}$. 

Now to fill the gap of $y \in  \left( 0, \frac{\mu r^k}{4}\right]$ we integrate $u_{yy}$ with respect to $y$ and use \textit{(C)} of Proposition \ref{semiconcavity}. For $(x,t) \in \overline{B}^*_{\frac{\mu r^k}{2}}(x_0) \times \left[t_0 - \frac{(\bar{C} \mu r^k)^2}{2}, t_0 \right]$ we have
\begin{align*}
u_y\left(x,\frac{\mu r^k}{2},t\right)-u_y(x,y,t) =\int_y^{\frac{\mu r^k}{2}} u_{yy} (x,h,t) \ dh \leq C_0 \frac{\mu r^k}{2} -C_0 y
\end{align*} 
where $C_0>0$ the constant of Proposition \ref{semiconcavity}. Then $u_y(x,y,t) \geq -C\theta^k+\frac{\varepsilon C \theta^k}{2}-C_0 \frac{\mu r^k}{2}$. 

Therefore in $\overline{B}^*_{\frac{\mu r^k}{2}}(x_0) \times \left(0, \frac{3\mu r^k}{4}\right] \times \left[t_0 - \frac{(\bar{C} \mu r^k)^2}{2}, t_0 \right]$ we have that $u_y(x,y,t) \geq -C\theta^k+\frac{\varepsilon C \theta^k}{2}-C_0 \mu r^k.$ We choose $0<r<\min \left\lbrace \frac{\mu}{2}, \frac{\bar{C}\mu}{\sqrt{2}} \right\rbrace<\frac{1}{2}$ then the above holds in $\overline{B}^*_{r^{k+1}}(x_0) \times ( 0, r^{k+1} ) \times \left[ t_0 - ( r^{k+1})^2, t_0 \right]$. Also using that $r<\theta$ and choosing $\mu < \frac{C\varepsilon}{4C_0}$ and $\theta>1-\frac{\varepsilon}{4}$ we have that $-C\theta^k+\frac{\varepsilon C \theta^k}{2}-C_0 \mu r^k \geq -C \theta^{k+1}$ and the induction is complete.

Taking $y \to 0^+$ in (\ref{sigmaregularity2}) we have that for any $k \in \N$
$$\sigma(x,t) \geq - C \theta^k, \ \ \text{ for every } \ \ (x,t) \in Q^*_{r^k}(x_0,t_0)$$
where $0<r <<\theta<1$ and $C>0$ universal. The desired regularity for $\sigma$ follows in a standard way. 
\end{proof}

We are ready now to obtain the proof of the main theorem of this work.

\begin{proof}[Proof of Theorem \ref{main_thin}]
First we use Lemma \ref{sigmaregularity} to get the regularity of $\sigma$ around $P_0 \in \Gamma^* \cap Q^*_{1/2}$. So Lemma \ref{sigmaregularity} gives that $\sigma (x_0,t_0)=0$. Indeed we know that $\sigma=0$ in $\Omega^*$ and since $\Gamma^* = \partial \Omega^* \cap Q_1^*$ there exists $\{ (x_k,t_k) \}_{k \in \N} \subset \Omega^* \cap \overline{Q}_{1/2}^*$ so that $(x_k,t_k) \to (x_0,t_0)$ as $k \to \infty$. We have $0 \geq \sigma(x_0,t_0) \geq -C \left( |x_0-x_k|+|t_0-t_k|^{1/2}\right)^\alpha$ for any large $k \in \N$. Thus taking $k \to \infty$ we get the desired. In addition we have that $0 \geq \sigma(x,t) \geq -C \left( |x-x_0|+|t-t_0|^{1/2}\right)^\alpha$, for any $(x,t) \in \overline{Q}_{1/4}^*(x_0,t_0)$. Indeed, we consider again $\{ (x_k,t_k) \}_{k \in \N} \subset \Omega^* \cap \overline{Q}_{1/2}^*$ so that $(x_k,t_k) \to (x_0,t_0)$ as $k \to \infty$. We have $0 \geq \sigma(x,t) \geq -C \left( |x-x_k|+|t-t_k|^{1/2}\right)^\alpha$ for any large $k \in \N$ and any $(x,t) \in \overline{Q}_{1/4}^*(x_0,t_0)$ and we let $k \to \infty$.

On the other hand we know that $u_y=\sigma$ on $Q_1^*$ in the classical sense (thus, in the viscosity sense as well). Then once the Neumann data $\sigma$ is $H^\alpha$ we can apply Theorem 17 of \cite{CM19} in $\overline{Q}_{1/4}^+(x_0,t_0)$ to complete the proof.
\end{proof}

\section*{Acknowledgments}
This work is part of my Ph.D thesis. I would like to thank my thesis advisor, Professor Emmanouil Milakis for his guidance and fruitful discussions regarding the topics of this paper. This work was co-funded by the European Regional Development Fund and the Republic of Cyprus through the
Research and Innovation Foundation (Project: EXCELLENCE/1216/0025).

\bibliographystyle{plain}   % Here the bibliography
\bibliography{biblio}             % is inserted.
\index{Bibliography@\emph{Bibliography}}%

\vspace{3em}

%\begin{tabular}{l}
%Georgiana Chatzigeorgiou\\ University of Cyprus \\ Department of Mathematics \& Statistics \\ P.O. Box 20537\\
%Nicosia, CY- 1678 CYPRUS
%\\ {\small \tt chatzigeorgiou.georgiana@ucy.ac.cy}
%\end{tabular}

%
%\begin{tabular}{l}
%Name\\
%Address 1, \\
%Address 1 \\
%Address 1  \\
%{e-mail : \small \tt email@email}
%\hfill
%\end{tabular}

\end{document}